\documentclass[12pt,reqno]{amsart}
\usepackage{amsmath}
\usepackage{amssymb}
\usepackage[left=3cm,top=3cm,right=3cm,bottom=3cm]{geometry}
\usepackage{epsfig}
\usepackage[normalem]{ulem}
\usepackage[usenames,dvipsnames]{color}
\usepackage{todonotes}

\usepackage{enumerate}

\numberwithin{equation}{section}

\newcommand{\remove}[1]{}

\begin{document}
\newtheorem{theorem}{Theorem}[section]
\newtheorem{lemma}[theorem]{Lemma}
\newtheorem{sublemma}[theorem]{Sub-lemma}
\newtheorem{definition}[theorem]{Definition}
\newtheorem{conjecture}[theorem]{Conjecture}
\newtheorem{proposition}[theorem]{Proposition}
\newtheorem{claim}[theorem]{Claim}
\newtheorem{algorithm}[theorem]{Algorithm}
\newtheorem{corollary}[theorem]{Corollary}
\newtheorem{observation}[theorem]{Observation}
\newtheorem{problem}[theorem]{Open Problem}
\newcommand{\R}{{\mathbb R}}
\newcommand{\N}{{\mathbb N}}
\newcommand{\Z}{{\mathbb Z}}
\newcommand\eps{\varepsilon}
\newcommand{\E}{\mathbb E}
\newcommand{\Prob}{\mathbb{P}}
\newcommand{\pl}{\textrm{C}}
\newcommand{\dang}{\textrm{dang}}
\renewcommand{\labelenumi}{(\roman{enumi})}
\newcommand{\bc}{\bar c}
\newcommand{\cal}[1]{\mathcal{#1}}
\newcommand{\G}{{\cal G}}
\newcommand{\Hc}{{\cal H}}
\newcommand{\Gnd}{\G_{n,d}}
\newcommand{\Gnp}{\G(n,p)}
\renewcommand{\P}{{\cal P}}
\newcommand{\la}{\lambda}
\newcommand{\floor}[1]{\lfloor #1 \rfloor}

\newcommand{\bel}[1]{\be\label{#1}}
\newcommand{\ee}{\end{equation}}
\newcommand{\be}{\begin{equation}}
 \newcommand\eqn[1]{(\ref{#1})}
 \newcommand{\ex}{\E}

\newcommand{\aas}{{a.a.s.}}
\newcommand{\wO}{\widetilde O}
\newcommand{\accessconst}{\gammaconst}
\newcommand{\gammaconst}{9}

\newcommand{\Aconst}{a} % setting it as A clashes with A subset of sphere.
\newcommand{\Bconst}{b} 
\newcommand{\hatU}{\widehat U}
\newcommand{\Bin}{{\rm Bin}}
\newcommand{\tildeU}{{\widetilde U}}

\newcommand{\norm}[1]{\left\lVert#1\right\rVert}

% Calum's macros
\newcommand{\scr}{\mathcal}
\newcommand{\mb}{\mathbb}
\newcommand{\til}{\widetilde}

\newcommand{\mil}{\mathit}

\newcommand{\opt}{\textup{opt}}
\newcommand{\lmax}{\textup{lmax}}
\newcommand{\pred}{\textup{pred}}

\newcommand{\Avg}{\text{Avg}}

\newcommand{\schur}{\textup{schur}}

\newcommand{\Enc}{\textup{Enc}}
\newcommand{\Dec}{\textup{Dec}}
\newcommand{\mPi}{\mathit{\Pi}}
\newcommand{\gam}{\gamma}
\newcommand{\Gam}{\Gamma}

\newcommand{\disc}{\operatorname{disc}}

\title{Tight Bounds on Probabilistic Zero Forcing on Hypercubes and Grids} 

\author{Natalie C. Behague}
\address{Department of Mathematics, Ryerson University, Toronto, ON, Canada}
%\thanks{The first author was supported by NSERC and Ryerson University}
\email{\texttt{nbehague@ryerson.ca}}

\author{Trent Marbach}
\address{Department of Mathematics, Ryerson University, Toronto, ON, Canada}
\email{\tt trent.marbach@ryerson.ca}

\author{Pawe\l{} Pra\l{}at}
\address{Department of Mathematics, Ryerson University, Toronto, ON, Canada}
%\thanks{The first author was supported by NSERC and Ryerson University}
\email{\texttt{pralat@ryerson.ca}
}

%\subjclass{05C80, 05C57}

\maketitle

\begin{abstract}
Zero forcing is a deterministic iterative graph colouring process in which vertices are coloured either blue or white, and in every round, any blue vertices that have a single white neighbour force these white vertices to become blue. Here we study probabilistic zero forcing, where blue vertices have a non-zero probability of forcing each white neighbour to become blue. We explore the propagation time for probabilistic zero forcing on hypercubes and grids. 
\end{abstract}

%%%%%%%%%%%%%%%%%%%%%%%%%%%%%%%%%%%%%%%%%%%%%%%%%%%%%%%%%%%
\section{Introduction\label{intro}}
%%%%%%%%%%%%%%%%%%%%%%%%%%%%%%%%%%%%%%%%%%%%%%%%%%%%%%%%%%%

\subsection{Definition of Zero Forcing}

Zero forcing is an iterative graph colouring procedure which can model certain real-world propagation and search processes such as rumor spreading. Given a graph $G$ and a set of marked, or blue, vertices $Z\subseteq G$, the process of zero forcing involves the application of the \emph{zero forcing colour change rule} in which a blue vertex $u$ forces a non-blue (white) vertex $v$ to become blue if $N(u)\setminus Z=\{v\}$, that is, $u$ forces $v$ to become blue if $v$ is the only white neighbour of $u$.

We say that $Z$ is a \emph{zero forcing set} if when starting with $Z$ as the set of initially blue vertices, after iteratively applying the zero forcing colour change rule until no more vertices can be forced blue, the entire vertex set of $G$ becomes blue. Note that the order in which forces happen is arbitrary since if $u$ is in a position in which it can force $v$, this property will not be destroyed if other vertices are turned blue. As a result, we may process vertices sequentially (in any order) or all vertices that are ready to turn blue can do so simultaneously. The \emph{zero forcing number}, denoted $z(G)$, is the cardinality of the smallest zero forcing set of $G$.

\medskip

Zero forcing has sparked a lot of interest recently. Some work has been done on calculating or bounding the zero forcing number for specific structures such as graph products~\cite{AIM2008}, graphs with large girth~\cite{DK2015} and random graphs~\cite{BBEMP2018,KKB2019}, while others have studied variants of zero forcing such as connected zero forcing~\cite{BH2017} or positive semi-definite zero forcing~\cite{BBFHHSVV2010}.

While zero forcing is a relatively new concept (first introduced in~\cite{AIM2008}), the problem has many applications to other branches of mathematics and physics. For example, zero forcing can give insight into linear and quantum controllability for systems that stem from networks. More precisely, in~\cite{BDHSY2013}, it was shown that for both classical and quantum control systems that can be modelled by networks, the existence of certain zero forcing sets guarantees controllability of the system.

Another area closely related to zero forcing is power domination~\cite{ZKC2006}. Designed to model the situation where an electric company needs to continually monitor their power network, one method that is used is to place phase measurement units (PMUs) periodically through their network. To reduce the cost associated with this, one asks for the least number of PMUs necessary to observe a specific network fully. To be more specific, given a network modelled with a simple graph $G$, a PMU placed at a vertex will be able to observe every adjacent vertex. Furthermore, if an observed vertex has exactly one unobserved neighbour, this observed vertex can observe this neighbour. In this way, power domination involves an observation rule compatible with the zero forcing colour change rule.

\medskip

In the present paper we are concerned with a parameter associated with zero forcing known as the \emph{propagation time}, which is the fewest number of rounds necessary for a zero forcing set of size $z(G)$ to turn the entire graph blue. More formally, given a graph $G$ and a zero forcing set $Z$, we generate a finite sequence of sets $Z_0\subsetneq Z_1\subsetneq \dots\subsetneq Z_t$, where $Z_0=Z$, $Z_t=V(G)$, and given $Z_i$, we define $Z_{i+1}=Z_i\cup Y_i$, where $Y_i\subseteq V(G)\setminus Z_i$ is the set of white vertices that can be forced in the next round if $Z_i$ is the set of the blue vertices. Then the propagation time of $Z$, denoted $pt(G,Z)$, is defined to be $t$. The propagation time of the graph $G$ is then given by $pt(G)=\min_{Z} pt(G,Z)$, where the minimum is taken over all zero forcing sets $Z$ of cardinality $z(G)$. The propagation time for zero forcing has been studied in~\cite{HHKMWY2012}.

\subsection{Definition of Probabilistic Zero Forcing}

Zero forcing was initially formulated to bound a problem in linear algebra known as the min-rank problem~\cite{AIM2008}. In addition to this application to mathematics, zero forcing also models many real-world propagation processes. One specific application of zero forcing could be to rumor spreading, but the deterministic nature of zero forcing may not fit the chaotic nature of real-life situations. As such, probabilistic zero forcing has also been proposed and studied where blue vertices have a non-zero probability of forcing white neighbours, even if there is more than one white neighbour. More specifically, given a graph $G$, a set of blue vertices $Z$, and vertices $u\in Z$ and $v\in V(G)\setminus Z$ such that $uv\in E(G)$, in a given time step, vertex $u$ will force vertex $v$ to become blue with probability
\[
\mb{P}(u\text{ forces }v)=\frac{|N[u]\cap Z|}{\deg(u)},
\]
where $N[u]$ is the closed neighbourhood of $u$.

In a given round, each blue vertex will attempt to force each white neighbour independently. If this happens, we may say that the edge $uv$ is forced. A vertex becomes blue as long as it is forced by at least one blue neighbour, or in other words if at least one edge incident with it is forced. Note that if $v$ is the only white neighbour of $u$, then with probability $1$, $u$ forces $v$, so given an initial set of blue vertices, the set of vertices forced via probabilistic zero forcing is always a superset of the set of vertices forced by traditional zero forcing. In this sense, probabilistic zero forcing and traditional zero forcing can be coupled. In the context of rumor spreading, the probabilistic colour change rule captures the idea that someone is more likely to spread a rumor if many of their friends have already heard the rumor.

Under probabilistic zero forcing, given a connected graph, it is clear that starting with any non-empty subset of blue vertices will with probability 1 eventually turn the entire graph blue, so the zero forcing number of a graph is not an interesting parameter to study for probabilistic zero forcing. Initially in~\cite{KY2013}, the authors studied a parameter that quantifies how likely it is for a subset of vertices to become a traditional zero forcing set the first time-step that it theoretically could under probabilistic zero forcing.

Instead, in this paper, we will be concerned with a parameter that generalizes the zero forcing propagation time. This generalization was first introduced in~\cite{GH2018}. Given a graph $G$, and a set $Z\subseteq V(G)$, let $pt_{pzf}(G,Z)$ be the random variable that outputs the propagation time when probabilistic zero forcing is run with the initial blue set $Z$. For ease of notation, we will write $pt_{pzf}(G,v)=pt_{pzf}(G,\{v\})$. The propagation time for the graph $G$ will be defined as the random variable $pt_{pzf}(G)=\min_{v\in V(G)}pt_{pzf}(G,v)$. More specifically, $pt_{pzf}(G)$ is a random variable for the experiment in which $n$ iterations of probabilistic zero forcing are performed independently, one for each vertex of $G$, then the minimum is taken over the propagation times for these $n$ independent iterations.

\subsection{Asymptotic Notation}

Our results are asymptotic in nature, that is, we will assume that $n\to\infty$. Formally, we consider a sequence of graphs $G_n=(V_n,E_n)$ (for example, $G_n$ is an $n$-dimensional hypercube or $n$ by $n$ grid) and we are interested in events that hold \emph{asymptotically almost surely} (\emph{a.a.s.}), that is, events that hold with probability tending to 1 as $n\to \infty$.

Given two functions $f=f(n)$ and $g=g(n)$, we will write $f(n)=O(g(n))$ if there exists an absolute constant $c \in \R_+$ such that $|f(n)| \leq c|g(n)|$ for all $n$, $f(n)=\Omega(g(n))$ if $g(n)=O(f(n))$, $f(n)=\Theta(g(n))$ if $f(n)=O(g(n))$ and $f(n)=\Omega(g(n))$, and we write $f(n)=o(g(n))$ or $f(n) \ll g(n)$ if $\lim_{n\to\infty} f(n)/g(n)=0$. In addition, we write $f(n) \gg g(n)$ if $g(n)=o(f(n))$ and we write $f(n) \sim g(n)$ if $f(n)=(1+o(1))g(n)$, that is, $\lim_{n\to\infty} f(n)/g(n)=1$.

Finally, as typical in the field of random graphs, for expressions that clearly have to be an integer, we round up or down but do not specify which: the choice of which does not affect the argument.

\subsection{Results on Probabilistic Zero Forcing}

In~\cite{GH2018}, the authors studied probabilistic zero forcing, and more specifically the expected propagation time for many specific structures. A summary of this work is provided in the following theorem.  
\begin{theorem}[\cite{GH2018}]
Let $n>2$. Then
\begin{itemize}
    \item $\min_{v\in V(P_n)}\E(pt_{pzf}(P_n,v))=\begin{cases}n/2+2/3&\text{ if }n\text{ is even}\\
    n/2+1/2&\text{ if }n\text{ is odd},
    \end{cases}$
    \item$\min_{v\in V(C_n)}\E(pt_{pzf}(C_n,v))=\begin{cases}n/2+1/3&\text{ if }n\text{ is even}\\
    n/2+1/2&\text{ if }n\text{ is odd},
    \end{cases}$
    \item$\min_{v\in V(K_{1,n})}\E(pt_{pzf}(K_{1,n},v))=\Theta(\log n)$,
    \item$\Omega(\log\log n)=\min_{v\in V(K_n)}\E(pt_{pzf}(K_n,v))=O(\log n)$.
\end{itemize}
\end{theorem}

In~\cite{CCGHLOR2019}, the authors used tools developed for Markov chains to analyze the expected propagation time for many small graphs. The authors also showed, in addition to other things, that $\min_{v\in V(K_n)}\E(pt_{pzf}(K_n,v))=\Theta(\log\log n)$ and for any connected graph $G$, $\min_{v\in V(G)}\E(pt_{pzf}(G,v))=O(n)$. This result was then improved in~\cite{NS2019}, where the authors showed that 
\[
\log_2\log_2(n)\leq \min_{v\in V(G)}\E(pt_{pzf}(G,v))\leq \frac{n}2+o(n)
\] 
for general connected graphs $G$. In the same paper, the authors also showed that
\begin{equation}\label{eqn:radius_theorem}
    \min_{v \in V(G)} \E( pt_{pzf}(G,v)) = O(r \log(n/r)),
\end{equation}
where $r\ge 1$ denotes the radius of the connected graph $G$. Moreover, they provided a class of graphs for which the bounds of the theorem are tight. 

\medskip

In addition to the results mentioned above, the authors of~\cite{GH2018} also considered the \emph{binomial random graph} $\Gnp$, proving the following theorem. 

\begin{theorem}[\cite{GH2018}]\label{theorem ept random graph}
Let $0<p<1$ be constant. Then \aas~we have that
\[
\min_{v\in V(\Gnp)}\E(pt_{pzf}(\Gnp,v))=O((\log n)^2).
\]
\end{theorem}

\noindent However, the authors in \cite{NS2019} conjectured that for the random graph, a.a.s.\ 
$$
\min_{v\in V(\Gnp)}\E(pt_{pzf}(\Gnp,v))=(1+o(1))\log\log n. 
$$
Of course, Theorem~\ref{theorem ept random graph} can be improved immediately via~\eqref{eqn:radius_theorem} and the fact that for $0<p<1$ constant, the radius of $\Gnp$ is a.a.s. $2$ (see e.g.~\cite{FK2016}), but even with this improvement, the bound is still far from the conjectured value. In~\cite{EMP2021}, the authors explored probabilistic zero forcing on $\Gnp$ in more detail and, in particular, proved the above conjecture. Their results can be summarized in the following theorem that shows that probabilistic zero forcing occurs much faster in $\Gnp$ than in a general graph $G$, as evidenced by the bounds of~\eqref{eqn:radius_theorem}.

\begin{theorem}[\cite{EMP2021}]
Let $v\in V(\Gnp)$ be any vertex of $\Gnp$. \\
If $p = \log^{-o(1)} n$ (in particular, if $p$ is a constant), then a.a.s.\ 
$$pt_{pzf}(\Gnp,v) \sim \log_2 \log_2 n.$$
On the other hand, if $\log n / n \ll p = \log^{- \Omega(1)} n$, then a.a.s. 
$$pt_{pzf}(\Gnp,v) = \Theta( \log (1/p) ).$$
\end{theorem}

\medskip

The results of most interest to us are from~\cite{HS2020}. The authors of that paper are concerned with hypercubes $Q_n$ and grids $G_{m\times n}$, and prove the following result. 

\begin{theorem}[\cite{HS2020}]
The following bounds hold:
\begin{itemize}
\item $\min_{v\in V(Q_n)} \E(pt_{pzf}(Q_n,v)) = O(n \log n)$,
\item $(1/2+o(1)) (m+n) \le \min_{v\in V(G_{m\times n})} \E(pt_{pzf}(G_{m\times n},v)) \le (4+o(1)) (m+n)$.
\end{itemize}
\end{theorem}

\subsection{Our Results}

In this paper, we improve both results from~\cite{HS2020}. Instead of considering the expectation of the propagation time, we will calculate bounds on the propagation time that \aas~hold. The lower bounds for both families of graphs are trivial so the improvement is for the upper bounds.

\begin{theorem}\label{thm:main}
The following bounds hold \aas:
\begin{itemize}
\item [(a)] $pt_{pzf}(Q_n) \sim n$,
\item [(b)] $(1+o(1)) (m+n)/2 \le pt_{pzf}(G_{m\times n}) \le (1+10^{-7}) (m+n)/2$.
\end{itemize} 
\end{theorem}
\medskip

The bounds for $Q_n$ are asymptotically tight. Unfortunately, we could not prove the matching upper bound for $G_{m\times n}$ but our upper bound is very tight, and so it supports the conjecture that \aas\ $pt_{pzf}(G_{n\times n}) \sim n$. This conjecture is supported by independent simulations performed in~\cite{HS2020} as well as our own.  

%%%%%%%%%%%%%%%%%%%%%%%%%%%%%%%%%%%%%%%%%%%%%%%%%%%%%%%%%%%
\section{Preliminaries}
%%%%%%%%%%%%%%%%%%%%%%%%%%%%%%%%%%%%%%%%%%%%%%%%%%%%%%%%%%%

\subsection{Chernoff inequality}

Let us first state a specific instance of Chernoff's bound that we will find useful. Let $X \in \textrm{Bin}(n,p)$ be a random variable with the binomial distribution with parameters $n$ and $p$. Then, a consequence of \emph{Chernoff's bound} (see e.g.~\cite[Corollary~2.3]{JLR}) is that 
\begin{equation}\label{chern}
\Prob( |X-\E X| \ge \eps \E X) ) \le 2\exp \left( - \frac {\eps^2 \E X}{3} \right)  
\end{equation}
for  $0 < \eps < 3/2$. Moreover, let us mention that the bound holds for the general case in which $X=\sum_{i=1}^n X_i$ and $X_i \in \textrm{Bernoulli}(p_i)$ with (possibly) different $p_i$ (again, e.g.~see~\cite{JLR} for more details).

\subsection{Hoeffding--Azuma inequality}

Let $X_0, X_1, \ldots$ be an infinite sequence of random variables that make up a martingale; that is, for any $a \in \N$ we have $\E [X_{a} | X_{a-1}] = X_{a-1}$. Suppose that there exist constants $c_a > 0$ such that $|X_a-X_{a-1}| \le c_a$ for each $a \le t$. Then, the \emph{Hoeffding--Azuma inequality} implies that for every $b > 0$,
\begin{equation}\label{eq:h-a}
\Prob \Big( \exists i (0 \le i \le t) : |X_i - X_0| \ge b \Big) \le 2\exp \left( - \frac {b^2}{2 \sum_{a=1}^t c_a^2} \right).
\end{equation}

\subsection{Chernoff--Heoffding bounds for Markov chains}

This result is due to Chung, Lam, Liu and Mitzenmacher~\cite{CLLM2012}. Let $M$ be a discrete time ergodic Markov chain with state space $[n]$ and the stationary distribution $\pi$. $M$ may be interpreted as either the chain itself or the corresponding $n$ by $n$ transition matrix. The \emph{total variation distance} between $u$ and $w$, two distributions over $[n]$, is defined as 
$$
\norm{u-w}_{TV} = \max_{A \subseteq [n]} \left| \sum_{i \in A} u_i - \sum_{i \in A} w_i \right|.
$$
For any $\eps>0$, the \emph{mixing time} of Markov chain $M$ is defined as
$$
T(\eps) = \min \left\{ t \in \N : \max_x \norm{x M^t - \pi}_{TV} \le \eps \right\}, 
$$
where $x$ is an arbitrary initial distribution. We will also need a definition of the $\pi$-norm. The \emph{inner product under the $\pi$-kernel} is defined as $\langle u,v \rangle_{\pi} = \sum_{i \in [n]} u_i v_i / \pi_i$. Then, the \emph{$\pi$-norm} of $u$ is defined as $\norm{u}_{\pi} = \sqrt{ \langle u,u \rangle_{\pi} }$. Note that, in particular, $\norm{\pi}_{\pi} = 1$.

Now we are ready to state the main result from~\cite{CLLM2012}. Let $T = T(\eps)$ be the $\eps$-mixing time of Markov chain $M$ for $\eps \le 1/8$. Let $(V_1, V_2, \ldots , V_t)$ denote a $t$-step random walk on $M$ starting from an initial distribution $\varphi$ on $[n]$. For every $i \in [t]$, let $f_i : [n] \rightarrow [0, 1]$ be a weight function at step $i$ such that the expected weight $\E_{v\leftarrow \pi}[f_i(v)] = \mu$ for all $i$.  Define the total weight of the walk by $X = \sum^t_{i=1} f_i(V_i)$. There exists some constant $c$ (which is independent of $\mu$, $\delta$ and $\epsilon$) such that for $0 \le \delta \le 1$
\begin{equation}\label{eqn:chern_Markov}
    \Prob\Big( (1 - \delta)\mu t \le X \le (1 + \delta)\mu t \Big) \ge 1 - c\norm{\varphi}_{\pi} \exp{\left(-\delta^2 \mu t / (72T)\right)}.
\end{equation}

\subsection{Useful Coupling} 

We will be using the following obvious coupling to simplify both upper and lower bounds. Having said that, in our case we will only need to prove upper bounds. Indeed, for lower bounds, it might be convenient to make some white vertices blue at some point of the process. Similarly, for upper bounds, one might want to make some blue vertices white. Given a graph $G$, $S,T\subseteq V(G)$, and $\ell \in \N$, let $A(S,T,\ell)$ be the event that starting with blue set $S$, after $\ell$ rounds every vertex in $T$ is blue. The following simple observation was proved in~\cite{EMP2021}.

\begin{lemma}[\cite{EMP2021}]\label{lem:coupling}
For all sets $S_1\subseteq S_2\subseteq V(G)$, $T\subseteq V(G)$, and $\ell\in \N$,
\[
\Prob(A(S_1,T,\ell))\leq \Prob(A(S_2,T,\ell)).
\]
\end{lemma}

%%%%%%%%%%%%%%%%%%%%%%%%%%%%%%%%%%%%%%%%%%%%%%%%%%%%%%%%%%%
\section{Hypercubes}
%%%%%%%%%%%%%%%%%%%%%%%%%%%%%%%%%%%%%%%%%%%%%%%%%%%%%%%%%%%

This section is devoted to proving part (a) of Theorem~\ref{thm:main}. 
Let us start with a formal definition of the hypercube. The \emph{$n$-dimensional hypercube} $Q_n$ has vertex set consisting of all binary strings of length $n$ and there is an edge between two vertices if and only if their binary strings differ in exactly one bit. For $0 \le k \le n$, \emph{level $k$} of the hypercube $Q_n$ is defined to be the set of all vertices whose binary strings contain exactly $k$ ones. Note that each vertex in level $k$ has $k$ neighbours in level $k-1$ and $n-k$ neighbours in level $k+1$.

\begin{proof}[Proof of Theorem~\ref{thm:main}(a)]
Trivially, for any vertex $v \in V(Q_n)$, we deterministically have that $pt_{pzf}(Q_n,v) \ge n$. Hence, in order to show that $pt_{pzf}(Q_n) \sim n$, it is enough to show that  \aas\ $pt_{pzf}(Q_n,v) \le n + o(n)$ for some vertex $v\in V(Q_n)$. Let $v$ be the vertex $(0,0,\ldots,0)$ on level $0$; in fact, since $Q_n$ is a vertex-transitive graph, $v$ can be any vertex. 
 
\medskip 

Suppose that for some integer $k$, $0 \le k < n/\ln^2 n$, the following property $\mathcal{P}(k)$ holds at some point of the process: all vertices at levels up to $k$ are blue (including level $k$). We will show that with probability $1-o(n^{-1})$ property $\mathcal{P}(k+1)$ holds after an additional $O(\ln n)$ rounds. Since trivially $\mathcal{P}(0)$ holds, by the union bound (over $n/\ln^2 n$ possible values of $k$) we will conclude that \aas\ $\mathcal{P}(n/\ln^2 n)$ holds after $O(n / \ln n) =o(n)$ rounds. 

Suppose that property $\mathcal{P}(k)$ holds from some $k$, $0 \le k < n/\ln^2 n$. By Lemma~\ref{lem:coupling}, we may assume that all vertices at level $k+1$ are white. It will be convenient to independently consider 3 phases after which property $\mathcal{P}(k+1)$ will hold with probability $1-o(n^{-1})$. The first phase lasts $320 \ln n$ rounds. The probability that a given vertex at level $k+1$ stays white during this phase is at most
$$
\left( 1 - \frac {k+1}{n} \right)^{320 \ln n} \le \exp \left( - \frac {320 (k+1) \ln n} {n} \right) \le 1 - \frac {160(k+1)\ln n}{n} =: p,
$$
since $k \le n/\ln^2 n$. Hence, for a given vertex $w$ at level $k$, the number of neighbours at level $k+1$ that turned blue during this phase is stochastically lower bounded by $\Bin(n-k, 1-p) \ge \Bin(n/2,1-p) \ni X$ with $\E [X] = (1-p)n/2 = 80(k+1)\ln n$. It follows from Chernoff's bound~\eqref{chern} (applied with $\eps=1/2$) that $w$ has at most $40(k+1) \ln n$ blue neighbours at level $k+1$ with probability at most
$$
2 \exp \left( - \frac {\E[X]}{12} \right) = \frac {2}{n^{(80/12)(k+1)}} = o \left( \frac {1}{n^{k+1}} \right).
$$ 
By the union bound (over at most $n^k$ vertices at level $k$), with probability $1-o(n^{-1})$, each vertex at level $k$ has at least $40(k+1) \ln n$ blue neighbours at level $k+1$ at the end of the first phase. As we aim for a statement that holds with probability $1-o(n^{-1})$, we may assume that this property holds once we enter the second phase.

The second phase lasts $\log_{5/4} n$ rounds. As before, let us concentrate on a given vertex $w$ at level $k$. Suppose that $w$ has $\ell$ blue neighbours for some integer $\ell$ such that $40(k+1) \ln n \le \ell \le n/2$ ($w$ has only $k$ neighbours at level $k-1$ so, of course, it includes neighbours at level $k+1$). The number of white neighbours of $w$ (at level $k+1$) that turned blue in one round of the process is stochastically lower bounded by $\Bin(n-\ell, (\ell+1)/n) \ge \Bin(n/2,(\ell+1)/n) \ni Y$ with $\E[Y] = (\ell+1)/2 \ge 20(k+1) \ln n$. We get from Chernoff's bound~\eqref{chern} (applied with $\eps=1/2$) that $Y \le \frac{\ell +1}{4}$  %$Y \le 10(k+1) \ln n$ 
with probability at most 
$$
2 \exp \left( - \frac {\E[Y]}{12} \right) = \frac {2}{n^{(20/12)(k+1)}} = o\left( \frac {1}{n^{k+1} \ln n} \right).
$$
By the union bound (over at most $n^k$ vertices at level $k$ and at most $\log_{5/4} n$ rounds), with probability $1-o(n^{-1})$, each vertex at level $k$ increases the number of blue neighbours by a multiplicative factor of $5/4$ each round, reaching at least $n/2$ blue neighbours by the end of the second phase. We may assume that this property holds once we enter the third phase.

The third (and last) phase lasts $3 \log_2 n$ rounds. This time, let us concentrate on a given white vertex $w$ at level $k+1$. This vertex has $k+1$ neighbours at level $k$, each of which has at least $n/2$ blue neighbours. Hence, vertex $w$ stays white by the end of this phase with probability at most
$$
\left( \left( \frac 12 \right)^{k+1} \right)^{3 \log_2 n} = 2^{-3(k+1) \log_2 n} = \frac {1}{n^{3(k+1)}} = o \left( \frac {1}{n^{k+2}} \right).
$$
By the union bound (over at most $n^{k+1}$ vertices at level $k+1$), with probability $1-o(n^{-1})$, all vertices at level $k+1$ turn blue by the end of this phase and so property $\mathcal{P}(k+1)$ holds.

\medskip

Since we aim for a statement that holds \aas, we may assume that $\mathcal{P}(n/\ln^2 n)$ holds after $o(n)$ rounds, and continue the process from there. We say that a vertex at layer $k$ is \emph{happy} if all but at most $\ell = \ell(n) = \ln^5 n$ neighbours at layer $k-1$ are blue. Similarly, a vertex at layer $k$ is \emph{very happy} if not only it is happy but also all but at most $\ell$ neighbours at layer $k+1$ are blue. (Note that a happy or even a very happy vertex might still be white.) Trivially, all vertices at levels up to $n/\ln^2 n+1$ are happy (including level $n/\ln^2 n+1$), and all vertices at levels before level $n/\ln^2 n$ are very happy.

Suppose that for some integer $k$, $n/\ln^2 n \le k \le n-1$, all vertices at levels up to $k$ are happy (including level $k$). We will show that after one single round all vertices at layer $k+1$ are going to be happy and all vertices at layer $k-1$ are going to be very happy with probability $1-o(n^{-1})$. By the union bound (over all possible values of $k$), we will get that \aas\ after less than $n$ rounds all vertices of the hypercube are going to be very happy. 

For simplicity, by Lemma~\ref{lem:coupling} we may assume that all vertices at level $k$ are white (despite the fact that they are happy). Let us concentrate on a given vertex $w$ at level $k$. Since $w$ is happy and all of its blue neighbours at level $k-1$ are happy too, the probability that $w$ stays white is at most 
$$
\left( 1 - \frac {k-\ell}{n} \right)^{k-\ell}  \le \exp \left( - (1+o(1)) \frac {k^2}{n}\right) =: p.
$$
Now, the probability that a given vertex at level $k+1$ is \emph{not} happy is at most
\begin{eqnarray*}
\binom{k+1}{\ell} p^\ell &\le& n^{\ell} \exp \left( - (1+o(1)) \frac {k^2}{n} \, \ell \right) \\
&=& \exp \left( \ell \ln n - (1+o(1)) \frac {k^2}{n} \, \ell \right) \\
&\le& \exp \left( \ln^6 n - (1+o(1)) n \ln n \right) \\
&=& o \left( \frac {1}{2^n \, n} \right).
\end{eqnarray*}
Then the property that all vertices at level $k+1$ are happy holds by the union bound (over at most $2^n$ vertices at level $k+1$). 

We may also show that all vertices at level $k-1$ are very happy. 
If $k \ge n-\ell$, then all vertices at level $k-1$ are trivially very happy as they have less than $\ell$ neighbours at level $k$, so there is nothing to show. 
If $k \le n-\ell$ (but still $k \ge n / \ln^2 n$), then a given vertex at level $k-1$ is \emph{not} very happy with probability at most
$$
\binom{n-(k-1)}{\ell} p^\ell \le n^{\ell} \exp \left( - (1+o(1)) \frac {k^2}{n} \, \ell \right) = o \left( \frac {1}{2^n \, n} \right),
$$
and the desired bound holds by the union bound (over at most $2^n$ vertices at level $k-1$).

\medskip

At this point we may assume that all vertices of the hypercube are very happy. It is easy to see that all remaining white vertices will turn blue in one single round. Indeed, since all vertices are very happy, the probability that some vertex stays white is at most 
\begin{eqnarray*}
2^n \left( 1 - \frac {n-2\ell}{n} \right)^{n-2\ell} &=& 2^n \left( \frac {2\ell}{n} \right)^{(1+o(1))n} \\
&=& \exp \left( O(n) - (1+o(1)) n \ln n \right) \\ 
&=& o(1).
\end{eqnarray*}
The proof is finished. 
\end{proof}

%%%%%%%%%%%%%%%%%%%%%%%%%%%%%%%%%%%%%%%%%%%%%%%%%%%%%%%%%%%
\section{Grids}
%%%%%%%%%%%%%%%%%%%%%%%%%%%%%%%%%%%%%%%%%%%%%%%%%%%%%%%%%%%

This section is devoted to proving part (b) of Theorem~\ref{thm:main}. 
Let us start with a formal definition of grids and tori. The \emph{$m$ by $n$ grid graph} $G_{m \times n}$ is the $mn$ vertex graph that is the Cartesian product of a path on $m$ vertices and a path on  $n$ vertices. 
The \emph{$m$ by $n$ torus graph} $T_{m \times n}$  is the Cartesian product of a cycle on $m$ vertices and a cycle on  $n$ vertices. 

We can restrict our focus to the square grid $G_{n \times n}$ or the square torus graph $T_{n \times n}$. For a non-square grid $G_{n \times m}$ with $n < m$, once a central $n \times n$ square is all blue, the two adjacent columns turn blue with probability one on the next time-step and so on. Thus  $pt\left(G_{n \times m}\right) \le pt\left(G_{n \times n}\right) + \left\lceil \frac{m-n}{2}\right\rceil$. A similar argument holds for the non-square torus. Hence, it will be straightforward to generalize the results for asymmetric cases which we will do once we are done with symmetric cases.
 
We define the \emph{origin} of $G_{n \times n}$ to be the central vertex if $n$ is odd and one of the four central vertices if $n$ is even (the choice can be made arbitrarily as all of them are the same up to symmetry). Since $T_{n \times n}$ is vertex transitive, we may fix any vertex of $T_{n \times n}$ to be the origin of $T_{n \times n}$. For a given positive integer $k < \frac{n}{2}$ we define the \emph{$k$-principal-square} to be the $2k+1$ by $2k+1$ sub-grid centred on the origin.

\medskip

Let $s = \left\lfloor n / \ln^2 n \right\rfloor$. We will work in phases, where in phase $i$ we start with all vertices in an $is$-principal-square being blue and end with all vertices in an $(i+1)s$-principal-square being blue. For simplicity, by Lemma~\ref{lem:coupling} we may assume that at the beginning of each phase the only vertices that are blue are the ones that belong to the corresponding principal-square. In fact,  during each  phase we will turn a few more vertices white (if needed) so that the process behaves more predictably. We aim to bound the number of time-steps it takes to go from the $is$-principal-square to the $(i+1)s$-principal square. As a result, in total there will be at most $\frac{1}{2} \ln^2{n}$ phases. Note that the only phase that potentially differs between the grid and the torus is the final phase.

Let us start with the following simple observation that will allow us to ignore a few initial and final phases. 

\begin{lemma}\label{lem:increase_square_by_one}
Let $k \ge 1$. Suppose that at time $t_0$ all vertices in a $k$-principal-square are blue. Then with probability $1-o\left(1/n\right)$ at time $t_0 + 6 \ln{n}$ all vertices in a $(k+1)$-principal-square are blue.
\end{lemma}
\begin{proof}
Let $k \ge 1$, and suppose that at time $t_0$ all vertices in a $k$-principal-square are blue. 
Note that vertices adjacent to a blue vertex with three blue neighbours turn blue with probability 1, and so all vertices adjacent to a non-corner vertex of the $k$-principal-square will turn blue at time $t_0 + 1$. There are twelve (eight if $k=1$) vertices in the $(k+1)$-principal square that do not turn blue deterministically: the four corner vertices and the eight (four if $k=1$) vertices adjacent to the corners. We will call these \emph{corner--adjacent}.

A corner--adjacent vertex has a blue neighbour at time $t_0$ and so the probability that it stays white in one time-step is at most $3/4$. The probability it remains white after $4\ln{n}$ steps is at most $\left(\frac{3}{4}\right)^{4\ln{n}}$ and so the probability that all corner--adjacent vertices are blue at time $t_0 + 4\ln{n}$ is at least $1 - 8\left(\frac{3}{4}\right)^{4\ln{n}} = 1 - o\left(1/n\right)$. 

Conditioning on a corner vertex having two blue neighbours at time $t_1 = t_0 + 4\ln{n}$, the probability it stays white after the next time-step is at most $(3/4)^2 = 9/16$. The probability it is white after a further $2\ln{n}$ time-steps is at most $\left(\frac{9}{16}\right)^{2\ln{n}}$. In particular, the probability that all four corner vertices are blue at time $t_1 + 2\ln{n}$ steps is at least $1 - 4\left(\frac{9}{16}\right)^{2\ln{n}} = 1 - o\left(1/n\right)$.
\end{proof}

Lemma \ref{lem:increase_square_by_one} tells us that with high probability any single phase takes $O\left(6s\ln{n}\right) = o(n)$ time-steps, so we may safely ignore what happens in the first five phases and the last phase.  In particular, our argument is the same whether we are run the process on the grid or the torus. Hence, we may focus on the square grid graph $G_{n \times n}$.

\medskip

Recall that $s = \lfloor n / \ln^2 n \rfloor$. Let $5 \le i < n/2s$ and suppose that at time $t_0$ all vertices in the $(is)$-principal-square are blue. By Lemma~\ref{lem:coupling}, we may assume that these are the only blue vertices at that point of the process. We will show that with probability $1 - o(1/n)$, at time $t_0 + (1+\eps+o(1)) s$ all vertices in the $(i+1)s$-principal-square are blue, where $\eps > 0$ will be an explicit, very small constant that we are not ready to introduce yet. 

For simplicity, we identify the vertices of $G_{n \times n}$ with pairs $(a,b)$ from the set 
$$
\left\{ -\left\lfloor \frac{n-1}{2} \right\rfloor, \ldots, \left\lfloor \frac{n}{2} \right\rfloor \right\} \times \left\{-\left\lfloor \frac{n-1}{2} \right\rfloor, \ldots, \left\lfloor \frac{n}{2} \right\rfloor \right\}
$$ 
in the natural way with the origin at $(0,0)$. In particular, note that a $k$-principal-square contains all vertices $(a,b)$ with $|a| \le k$ and $|b| \le k$. We will focus on the top-right quadrant where both coordinates are positive; by symmetry, the argument for the other quadrants will be exactly the same.

In order to control how the process progresses with time, we will pay attention to a few blue vertices at the top-right corner that are at the same distance from the origin. Hence, the following definition will be useful. For fixed positive integer $d$, we define a \emph{$d$-window} (\emph{rooted at vertex $(a-d+1,b)$}) to be a $d$-tuple of vertices 
$$
\Big( (a-d+1,b), (a-d+2,b-1), (a-d+3,b-2),\ldots,(a,b-d+1) \Big),
$$ 
where $d < a,b < n/2$.  Note that the distance of all vertices from such a $d$-window from the origin is $a+b-d+1$.

Blue vertices from a $d$-window at distance $\ell$ from the origin will most likely turn some other vertices at distance $\ell+1$ from the origin to be blue. If that happens, we will simply move the window appropriately and continue the process. In order to compute the transition probabilities between given configurations that might occur in a $d$-window, it will be convenient to assume that each blue vertex in the window has exactly two blue neighbours, namely the bottom and the left neighbours. We may do so based on the following observation. 

\begin{lemma}\label{lem:two_blue_nbrs_persists}
Suppose that at time $t$, a $k \times k$ sub-grid of vertices are all blue for some integer $k \ge 2$. Let $v$ be the top-right corner (blue) vertex of this sub-grid. If a neighbour of $v$ turns blue at time $t+1$, then it is the top-right corner of a $(k-1) \times (k-1)$ sub-grid of vertices that are all blue.
\end{lemma}

\begin{proof}
Consider the four vertices adjacent to $v$. Clearly, the vertex below $v$ and the vertex to the left of $v$ are each the top right corner of a $(k-1) \times (k-1)$ all blue grid. Let $u$ be the vertex to the right of $v$. The column of $k-2$ vertices directly below $u$ are all blue at time $t + 1$ with probability 1 (as each is adjacent to a blue vertex with three blue neighbours at time $t$). Thus, if $u$ turns blue at time $t+1$, then it is the top-right corner of a $(k-1) \times (k-1)$ blue sub-grid---see Figure~\ref{fig:lemma_square}. An analogous argument holds for the vertex above $v$.
\end{proof}

\begin{figure}[h]
    \centering
    \includegraphics{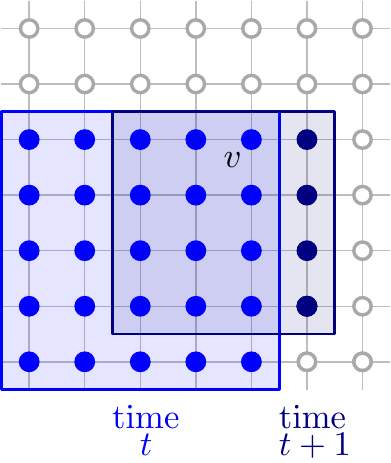}
    \caption{A $5\times 5$ blue sub-grid at time $t$ and a $4 \times 4$ blue sub-grid at time $t+1$ as described in Lemma \ref{lem:two_blue_nbrs_persists}.}
    \label{fig:lemma_square}
\end{figure}

Now, we are ready to define an \emph{auxiliary} process, which monitors the behaviour of the original process, giving a sequence of triples $(D_j,C_j,t_j)$ of $d$-windows $D_j$ at times $t_j$, and associated binary $d$-tuples $C_j$ indicating which vertices in the $d$-window are blue at time $t_j$ and which ones are white (1 represents blue and 0 represents white). This process will control the expansion of the blue principal squares and so can be used to upper bound the number of rounds needed to reach the end of a given phase. We will run this process for at most $3s$ steps and stop it \emph{prematurely} if the end of the phase is not reached after $3s$ rounds. However, we will show that we do \emph{not} stop prematurely with probability $1-o(1/n)$ and, in fact, when $d$ is large enough, with this probability we will finish in almost $s$ rounds which is a trivial lower bound for the number of rounds needed to finish a given phase.

\begin{enumerate}[(1)]
    \item We start with $D_0 = ( (is-d+1,is),\ldots,(is, is-d+1) )$ and time-step $t_0$. Since all vertices in $D_0$ are blue, the corresponding configuration is $C_0 = (1, \ldots, 1)$. (In fact, the initial triple is not important. Another natural starting point would be to start with a $d$-window including vertex $(is,is)$ and so would have only one blue vertex in the corresponding configuration.)
    \item Given the triple $(D_j,C_j,t_j)$, we define triple $(D_{j+1},C_{j+1},t_{j+1})$ differently according to whether  $C_j$ contains a blue vertex or not. Suppose that $D_j$ is a $d$-window rooted at vertex $(a-d+1,b)$.
    \item \label{item:D_j has blue vertex} If $C_j$ contains a blue vertex, let $t_{j+1} = t_j +1$ and consider the $d$-windows 
    \begin{eqnarray*} 
    X &=& ((a-d+1,b+1), \ldots,(a,b-d+2)) \quad \text{and} \\
    Y &=& ((a-d+2,b), \ldots,(a+1,b-d+1)).
    \end{eqnarray*} 
    Let $D_{j+1}$ be whichever of $X$ or $Y$ has more blue vertices at time $t_{j+1}$ (see Figure~\ref{fig:windowsXY} for an illustration). If they have an equal number of blue vertices, then pick one of the two uniformly at random. Let $C_{j+1}$ be the corresponding configuration capturing the information about which vertices in the window are blue at time $t_{j+1}$. Go to step (2).
    
    \begin{figure}[h]
        \centering
        \includegraphics{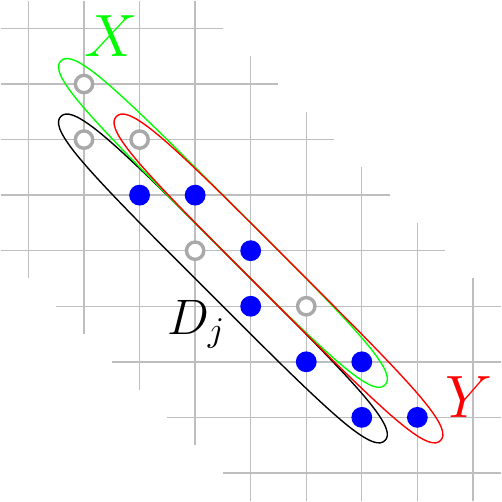}
        \caption{A 6-window $D_j$ and 6-windows $X$, $Y$ as defined in Step (3) of the auxiliary process. Since $Y$ has more blue vertices, $D_{j+1}=Y$.}
        \label{fig:windowsXY}
    \end{figure}
    
    \item \label{item:D_j all white} If  {$C_j$} is all white (that is, all vertices in $D_j$ are white at time $t_j$), we set $t_{j+1} = t_j$. In this situation we need to introduce an \emph{auxiliary} triple $(D_{j+1},C_{j+1},t_{j+1})$ that corresponds to an earlier $d$-window that is not all white at time $t_j$. To that end we chose $D_{j+1}$ to have the same distance from the origin as $D_{j-1}$. Note that $C_{j-1}$ is not all zeros and represents a configuration in the $d$-window $D_{j-1}$ at time $t_{j-1}$; indeed, the process is designed in such a way that no $C_{j-1}$ and $C_{j}$ can be all zeros at any step in this process. We may simply fix $D_{j+1} = D_{j-1}$ and $C_{j+1}=C_{j-1}$ but our goal is to create a memory-less Markov chain so we will follow a different strategy.
    
    Let $(x,y)$ be a random vertex in $D_{j-1}$ that was blue at time $t_{j-1}$. Let $D_{j+1}$ be the $d$-window centred on $(x,y)$. To be specific, if $d$ is odd then let 
    $$D_{j+1} = \left( \left(x - \frac{d-1}{2}, y + \frac{d-1}{2} \right), \ldots, \left(x  + \frac{d-1}{2}, y - \frac{d-1}{2}\right) \right), $$
    and if $d$ is even then let $D_{j+1}$ be one of 
    $$\left( \left(x - \frac{d-2}{2},y + \frac{d-2}{2}\right), \ldots, \left(x + \frac{d}{2},y - \frac{d}{2}\right) \right)$$ and 
    $$\left( \left( x - \frac{d}{2},y + \frac{d}{2} \right), \ldots, \left( x +\frac{d-2}{2},y - \frac{d-2}{2} \right) \right)$$ chosen uniformly at random. At time $t_{j-1}$, $(x,y)$ is blue and it follows from Lemma~\ref{lem:two_blue_nbrs_persists} that $(x-1,y)$ and $(x,y-1)$ are blue and have three blue neighbours. (See the discussion below for details about implications of Lemma~\ref{lem:two_blue_nbrs_persists}.) This means at time $t_{j+1} = t_j = t_{j-1}+1$, the vertices $(x-1,y+1),(x,y),(x+1,y-1)$ are each blue and in $D_{j+1}$. Applying Lemma~\ref{lem:coupling}, we may assume that any vertices in $D_{j+1}$ that are not one of $(x-1,y+1), (x,y), (x+1,y-1)$ are white at time $t_{j+1} = t_j$. Let us stress again that by our choice of $(x,y)$, $D_{j+1}$ is not all white at time $t_{j+1}$. Go to step (2).
\end{enumerate}

\medskip

Recall that we start the current phase at time $t_0$ with all vertices in the $(is)$-principal-square being blue. Moreover, we already dealt with the first few phases so $i \ge 5$. By this assumption, every vertex in the starting $d$-window $D_0$ is the top-right corner of a $(2is+1-d) \times (2is+1-d)$ blue sub-grid at time $t_0$. Applying  Lemma~\ref{lem:two_blue_nbrs_persists} recursively for each $j$, we see that if a vertex in $D_j$ turns blue at time $t_j$  because of its blue neighbour in $D_{j-1}$, then it is the top-right corner of a $(2is + 1 -d-j) \times (2is + 1-d-j)$ blue sub-grid. In particular, since we run the auxiliary process for at most $3s$ steps, it will be the top-right corner of a $3 \times 3$ blue sub-grid. 

\medskip

We are now ready to investigate the transition probabilities between configurations $C_j$ and guide the auxiliary process so that it yields a Markov chain. Suppose that $D_j = ((a-d+1,b), \ldots,(a,b-d+1))$. Let us first investigate step (\ref{item:D_j has blue vertex}) of the auxiliary process. By the assumption of this step, $C_j$ is not all zeros; that is, $D_j$ contains a blue vertex at time $t_j$. By Lemma~\ref{lem:coupling} we may assume that at time $t_j$ all vertices in the next $d$-windows $X = ((a-d+1,b+1), \ldots,(a,b-d+2))$ and $Y = ((a-d+2,b), \ldots,(a+1,b-d+1))$ are white, and that the vertices $(a-d,b+1)$ and $(a+1,b-d)$ are white. Moreover, as discussed above, Lemma~\ref{lem:two_blue_nbrs_persists} implies that any vertex in $D_j$ that is blue at time $t_j$ has exactly two blue neighbours. Thus the probability that a vertex in $X \cup Y$ turns blue at time $t_j + 1$ is exactly $1-(1/4)^2=15/16$ if it has two blue neighbours in $D_j$ at time $t_j$, $1-1/4=3/4$ if it has one blue neighbour in $D_j$ at time $t_j$, and $0$ otherwise. In particular, the configuration $C_{j+1}$ representing the state of $D_{j+1}$ at time $t_{j+1}$ depends only on the configuration $C_j$ representing the state of $D_j$ at time $t_j$. The transition probability $P_d(C,C')$ between any two possible configurations $C$ and $C'$ can be easily computed.

Now, let us investigate step~(\ref{item:D_j all white}). This is set up so that if $C_j$ is all white (that is, $D_j$ contains no blue vertices at time $t_j$) then deterministically $C_{j+1}$ consists of three centred consecutive blue vertices and all other vertices white. In particular, for $d \ge 3$ odd and $C$ all white, $P_d(C,C') = 1$ if $C'=(\underbrace{0,\ldots, 0}_{(d-3)/2}, 1, 1, 1,  \underbrace{0,\ldots, 0}_{(d-3)/2})$ and $0$ otherwise. For $d \ge 4$ even and $C$ all white, 
$$P_d(C,C') = 
    \begin{cases} 
        1/2 &\text{if } C' = (\underbrace{0,\ldots, 0}_{(d-4)/2}, 1, 1, 1,  \underbrace{0,\ldots, 0}_{(d-2)/2}) \text{  or } (\underbrace{0,\ldots, 0}_{(d-2)/2}, 1, 1, 1,  \underbrace{0,\ldots, 0}_{(d-4)/2})\\
        0 &\text{otherwise.}   
    \end{cases}
$$
Finally, for the degenerate case $d=2$ and $C$ all white, $P_2(C,C')=1$ if $C'$ is all blue, and 0 otherwise.

\medskip

Since the state of $D_{j+1}$ at time $t_{j+1}$ (namely, configuration $C_{j+1}$) depends only on the state of $D_j$ at time $t_j$ (namely, configuration $C_{j}$) independent of the time and the choice of vertices, we can capture the behaviour of the process using a Markov chain $(C_j)_{j=0}^{\infty}$ on state space $\Omega =  \{0,1\}^d$. 
For example, when $d=2$ the Markov chain has transition matrix
$$ P_2 =
\begin{pmatrix}
0 & 0 & 0 & 1\\
 1/16  &  9/32 &    3/32 &  9/16 \\
1/16  & 3/32 &  9/32 &  9/16 \\
1/256 & 15/256 & 15/256 & 225/256
\end{pmatrix}$$
where the states are in the order $(0, 0), (0,1),(1,0),(1,1)$.
It is easy to see that for any $d$ this Markov chain is irreducible and aperiodic, and so it has a limiting distribution which is equal to the stationary distribution. We define $\mu_d$ to be the probability of the all white state $C=(0,0,\ldots,0)$ in the limiting distribution. 

The next lemma is the crux of our argument. 

\begin{lemma} \label{lem:grids_main}
Take $s = \lfloor n / \ln^2 n \rfloor$. Let $5 \le i < n/(2s)$ and suppose that at time $t_0$ all vertices in the $(is)$-principal-square are blue. With probability $1 - o(1/n)$, before the end of round $t_0 + (1+o(1)) \frac{1-\mu_d}{1-2\mu_d}2s$ all vertices in the $(i+1)s$-principal-square are blue, where $\mu_d$ is the probability of the all white state $C=(0,0,\ldots,0)$ in the limiting stationary distribution of the Markov's chain $P_d$ defined above.
\end{lemma}

Applying the Lemma together with earlier observations we immediately obtain the following corollary.

\begin{corollary} \label{cor:main}
The following bound holds \aas:
$$
pt\left(G_{n \times n}\right) \le (1+o(1)) \frac{1-\mu_d}{1-2\mu_d} \, n,
$$
where $\mu_d$ is the probability of the all white state $C=(0,0,\ldots,0)$ in the limiting stationary distribution of the Markov's chain $P_d$ defined above.
\end{corollary}

\begin{proof}[Proof of Lemma \ref{lem:grids_main}]
Let us fix an integer $d \ge 2$. Suppose we run the auxiliary process $(D_j, C_j, t_j)$ monitoring the trajectory of the $d$-window for $\ell$ steps. Let $w$ be the number of steps we spend in the all white state $C=(0,0,\ldots,0)$. We have that $t_\ell = t_0 + \ell - w$, since when we are in the all white state $t_{j+1} = t_j$, and $t_{j+1} = t_j +1$ otherwise. The  distance of $D_0$ from the origin is $2(is) -d +1$. The distance of $D_\ell$ from the origin is $2(is)-d+1 + \ell - 2w$, since when we are  in the all white state $D_{j+1}$ is one closer to the origin than $D_j$,  and otherwise $D_{j+1}$ is one further from the origin than $D_j$.

Applying Chernoff--Heoffding bounds for Markov chains (equation~\eqref{eqn:chern_Markov}), we have that 
$$
\Prob \Big( (1 - \delta)\mu_d \ell \le w \le (1 + \delta)\mu_d \ell \Big) \ge 1 - c \exp{(- c\delta^2 \ell)}
$$
for $0 \le \delta \le 1$, where $c$ is some positive constant depending only on the Markov chain (in particular, independent of $\ell$ and $\delta$). Set 
$$
\ell = \left(\frac{2}{1-2\mu_d} + \frac{\ln^4{n}}{\sqrt{n}}\right)s \ge s \quad \text{ and } \quad \delta = \frac{\ln^2{n}}{\sqrt{n}}. 
$$
(Note that $\mu_d \le 1/77$ implies that $\ell < 3s$, which is the number of steps taken. We will show that $\mu_d$ satisfies this property below and so we may assume that the process does \emph{not} stop prematurely.) Then, using that $s = \lfloor n / \ln^2{n} \rfloor$, we have that with probability $1-o(1/n)$
the distance of $D_\ell$ from the origin is at least
\begin{align*}
2is - d+1 & + \ell - 2(1+\delta) \mu_d \ell \\
& = 2is + O(1) + (1-2\mu_d) \left( 1 - \Theta \left( \frac{ \ln^2 n }{ \sqrt{n} } \right) \right) \ell \\
& = 2is + O(1) + (1-2\mu_d) \left( 1 - \Theta \left( \frac{ \ln^2 n }{ \sqrt{n} } \right) \right) \frac {2}{1-2\mu_d} \left( 1 + \Theta \left( \frac{ \ln^4 n }{ \sqrt{n} } \right) \right) s \\
& = 2(i+1) s + \Theta( \sqrt{n} \ln^2 n),
\end{align*}
and the number of rounds in this phase that passed in the original zero-forcing process is equal to
\begin{eqnarray*}
t_\ell - t_0 = \ell - w \le (1 - (1-\delta) \mu_d ) \ell &=& (1 - \mu_d) \left( 1 + \Theta \left( \frac {\ln^2 n}{\sqrt{n}} \right) \right) \ell \\
&=& (1 + o(1))\frac{1-\mu_d}{1-2\mu_d}2s.
\end{eqnarray*}

\medskip

It remains to show that the $d$-windows $D_j$ are travelling broadly North-East rather than North or East, in order to be sure of obtaining a principal square as opposed to a rectangle. To that end, let us define the discrepancy $\disc(D)$ of a $d$-window $D = (a-d+1,b),\ldots,(a,b-d+1)$ to be $a-b$. The discrepancy captures how far from the North-East diagonal the $d$-window is shifted. The discrepancy of $D_0$ is $0$. By definition, $\disc\left(D_{j+1}\right)$ differs from $\disc\left(D_{j}\right)$ by at most one if $D_j$ is not all white at time $t_j$, and differs from $\disc\left(D_{j}\right)$ by at most $d$ if  $D_j$ is all white at time $t_j$.

Define a sequence $(j_k)_{k \ge 0}$ where $j_0 = 0$ and for $k > 0$, $j_k$ is the least $j > j_{k-1}$ such that $D_j$ is all white at time $t_j$ (that is, $C_j = (0, 0, \ldots, 0)$). For any $j$ for which $C_j \neq (0,0,\ldots,0)$, the probability that $C_{j+1} = (0,0,\ldots,0)$ is at least $(1/16)^{d+1}$. Hence, the probability that we do not hit an all white state after $2 \cdot 16^{d+1}\ln{n}$ consecutive steps, is at most
$$
\left(1 - (1/16)^{d+1}\right)^{ 2 \cdot 16^{d+1}\ln{n}} \le \exp ( -2 \ln n) = 1/n^2. 
$$
Thus, since there are at most $\ell = o(n)$ terms in the sequence, we have $j_k - j_{k-1} \le 2 \cdot 16^{d+1}\ln{n}$ for all $k$ with probability $1- o(1/n)$. Since we aim for a statement that holds with probability $1-o(1/n)$, we may assume that this property is deterministically satisfied.

Consider the sequence $( \disc\left(D_{j_k}\right) )_{k \ge 0}$. Based on our assumption, for all $k \ge 1$
$$|\disc\left(D_{j_k}\right) -\disc\left(D_{j_{k-1}}\right)| 
\le d + t_k - t_{k-1} \le \left( 2 \cdot 16^{d+1} + o(1) \right)\ln{n} \le 3 \cdot 16^{d+1} \ln n.
$$
By symmetry of the auxiliary process, we know that given $\disc\left(D_{j_{k-1}}\right)$ and integer $a$ the probability that $\disc\left(D_{j_k}\right) = \disc\left(D_{j_{k-1}}\right) + a$ must be the same as the probability that  $\disc\left(D_{j_k}\right) = \disc\left(D_{j_{k-1}}\right) - a$. Thus the sequence $(\disc\left(D_{j_k}\right) )_{k \ge 0}$ is a martingale. 
We can apply the Hoeffding-Azuma inequality~\eqref{eq:h-a} with $b=\sqrt{\ell} \, \ln^2 n$ and $2 \cdot 16^{d+1} \ln^2 n \le j_k \le \ell$ to see that 
$$
\Prob\left(|\disc\left(D_{j_k}\right)| \ge b \right) \le 2\exp\left(\frac{-b^2 }{2 \cdot k \cdot 3 \cdot 16^{d+1} \ln^2 {n}}\right) = 2\exp\left(\frac{-\ell \, \ln^4 n }{O(\ell \ln^2 n)}\right) = o(1/n),
$$
and so 
$$\Prob\left(|\disc\left(D_{\ell}\right)| \le \sqrt{\ell}\ln^2{n} + 3 \cdot 16^{d+1} \ln n \right) \ge 1 - o(1/n),$$
where the additional term $3 \cdot 16^{d+1} \ln n$ needed to be added because the last term in the sequence $(j_k)_{k \ge 0}$ can be smaller than $\ell$ but, with the desired probability, not smaller than $\ell - 3 \cdot 16^{d+1} \ln n$.

\medskip

Putting all of this together, we conclude that with probability $1 - o(1/n)$ at time $t_\ell = t_0 + (1 + o(1))\frac{1-\mu_d}{1-2\mu_d} 2s$ there is a blue vertex at distance $2(i+1)s + \Theta \left(\sqrt{n} \, \ln^2{n}\right)$ from the origin with a discrepancy of less than $\sqrt{\ell} \ln^2{n} = O\left(\sqrt{n} \, \ln{n} \right) = o\left(\sqrt{n} \, \ln^2{n} \right)$. Using Lemma~\ref{lem:two_blue_nbrs_persists} for the last time, we get that this blue vertex is the top-right corner of a $(2is + 1 -d-\ell) \times (2is + 1-d-\ell)$ all blue sub-grid at time $t_\ell$. Since $i \ge 5$, we get that
$$
2is + 1 -d - \ell \ge (i+1)s +  4s + 1 - d - 3s > (i+1)s,
$$ 
and so this all blue sub-grid entirely contains the top-right quadrant of the $(i+1)s$-principal-square.

By symmetry, the same conclusion holds for the other three quadrants and so with probability $1 - o(1/n)$ at time $t_0 + (1 + o(1))\frac{1-\mu_d}{1-2\mu_d} 2s$ the $(i+1)s$-principal-square is entirely blue. The proof of the lemma is finished.
\end{proof}

The only thing remaining to finish the main theorem is to calculate the stationary distribution of the Markov chain and thus $\mu_d$. Based on Corollary~\ref{cor:main}, each value of $\mu_d$ implies that \aas, 
$$
 pt\left(G_{n \times n}\right) \le (1 + \eps_d + o(1)) \, n, \quad \text{ where } \quad \eps_d := \frac{1-\mu_d}{1-2\mu_d} - 1 = \frac {\mu_d}{1-2\mu_d}.
$$
When $d = 2$ or $d=3$ the transition matrix is small enough to be calculated by hand and one can check that $\mu_2 = 1/77$ and $\mu_3 = 1861/491117$, respectively. It gives us $\eps_2 = 1/75 \le 0.01334$ and $\eps_3 = 1861/487395 \le 0.003819$.
For $d \le 7$, one can use a computer to calculate the exact fractions  in the transition matrix and thus the exact values of $\mu_d$. This gives 
\begin{align*}
    \mu_4 = & \frac{11439524}{9092101243}  \\ 
    \mu_5 = & \frac{1133763610798567}{2542177028478096119} \\ 
    \mu_6 = & \frac{112666827183116235892325831063}{686127236264864409019398540749761}  \\ 
    \mu_7 = & \frac{536778086928248989283123883507309287148693034345565}{8663791645046173690408989931892492266198652103814670581}. 
\end{align*}
For larger values of $d$, we must move to numerical approximations. Rounding errors become a concern for $d > 14$ when the minimum entry in the transition matrix is close to the precision of the computer. Below we summarize these numerical values in the table. In particular, when $d = 14$ we obtain $\eps_{14} < 10^{-7}$ and the main theorem holds.

\begin{center}
\begin{tabular}{ |l|l|l| } 
 \hline
 $d$ & $\mu_d$ & $\eps_d$ \\
 \hline
2 & $0.012987012987012988$ & $0.013333333333333334$ \\
3 & $0.0037893210782766634$ & $0.0038182582915294574$ \\
4 & $0.0012581826460420552$ & $0.001261356680213082$ \\
5 & $0.0004459813766302923$ & $0.0004463795305453566$ \\
6 & $0.00016420690103551534$ & $0.00016426084656466706$ \\
7 & $6.1956486134 \cdot 10^{-5}$ & $6.1964164298 \cdot 10^{-5}$ \\ 
8 & $2.3776197997 \cdot 10^{-5}$ & $2.3777328666 \cdot 10^{-5}$ \\
 9 & $9.2381456535 \cdot 10^{-6}$ & $9.2383163433 \cdot 10^{-6}$ \\
10 & $3.6235531968 \cdot 10^{-6}$ & $3.6235794573 \cdot 10^{-6}$ \\
11 & $1.4319129399 \cdot10^{-6}$ & $1.4319170406 \cdot 10^{-6}$ \\
12 & $5.6925354755 \cdot 10^{-7}$ & $5.6925419565 \cdot 10^{-7}$ \\
13 & $2.2742611942 \cdot 10^{-7}$ & $2.2742622287 \cdot 10^{-7}$ \\
14 & $9.1236746477 \cdot 10^{-8}$ & $9.1236763126 \cdot 10^{-8}$\\
 \hline
\end{tabular}
\end{center}

\section{Acknowledgment} 

This research has been partially supported by NSERC and by the Ryerson University Faculty of Science Dean’s Research Fund.
The numerical results presented in this paper were obtained using the Julia language~\cite{Julia}. We would like to thank Bogumi\l{} Kami\'nski from SGH Warsaw School of Economics for helping us to implement it. The program is available on-line.\footnote{\texttt{https://math.ryerson.ca/\~{}pralat/}} 

\bibliographystyle{plain}
\bibliography{ref}

\end{document}